\documentclass[11pt,a4paper]{article}
\usepackage[utf8x]{inputenc}
\usepackage[T1]{fontenc}

\title{The fibrewise trivial locus of a line bundle}
\author{Amira Tlemsani}
\usepackage[pdftex]{graphicx} 
\usepackage[pdftex,linkcolor=black,pdfborder={0 0 0}]{hyperref} 
\usepackage{enumitem} 
\usepackage{xcolor} 
\usepackage{amsmath}
\usepackage{amsthm}
\usepackage{amssymb}
\usepackage{tikz-cd}
\usepackage{mathrsfs}

\usepackage{cleveref}
\usepackage{mathtools}

\usepackage[a4paper, lmargin=0.1666\paperwidth, rmargin=0.1666\paperwidth, tmargin=0.1111\paperheight, bmargin=0.1111\paperheight]{geometry} 

\usepackage[all]{nowidow} 
\usepackage[protrusion=true,expansion=true]{microtype} 
\usepackage{extarrows}

\usepackage{lipsum} 

\usepackage[toc,page]{appendix}

\hypersetup{ 	
pdfsubject = {},
pdftitle = {},
pdfauthor = {}
}

\newtheorem{theorem}{Theorem}
\newtheorem{lemma}[theorem]{Lemma}
\newtheorem{definition}[theorem]{Definition}
\newtheorem{remark}[theorem]{Remark}
\newtheorem{example}[theorem]{Example}
\newtheorem{proposition}[theorem]{Proposition}

\newtheorem*{example*}{Example}
\newtheorem*{acknowledgements}{Acknowledgements}

\begin{document} 

\maketitle 

\begin{abstract}
    Let $\mathcal{L}$ be a line bundle on a smooth and proper scheme $X$ over $S$. We compute, in the case where $S$ is smooth over a field of characteristic $0$, the virtual fundamental class of the closed subset of $S$ consisting of those $s\in S$ such that the restriction of $\mathcal{L}$ to the fibre $X_s$ over $s$ is trivial.
\end{abstract}

\newcommand{\J}{\mathcal{J}}
\newcommand{\C}{\mathcal{C}}
\newcommand{\Lu}{\mathcal{L}^u}
\newcommand{\on}[1]{\operatorname{#1}}
\newcommand{\Pic}{\on{Pic}^0_{X/S}}
\newcommand{\DR}{\operatorname{DR}}
\newcommand{\OP}{\mathsf{OP}}

\section{Introduction}

Let $X/S$ be a smooth projective morphism of schemes, and let $\mathcal{L}$ be a line bundle on $X$. Writing $X_s$ for the fibre of $X$ over $s$, the closed subset where $\mathcal{L}$ is fibrewise trivial
$$\{s\in S: \mathcal{L}_{|X_s}\simeq \mathcal{O}_{X_s}\} \subseteq S$$
can be equipped with a virtual fundamental class, which we call \textit{the double ramification cycle} of $\mathcal{L}$. This can be defined by the formula
$$\on{DR}(\mathcal{L}) \coloneqq \sigma^{*}e$$
where $\sigma, e$ are the sections of $\on{Pic}_{X/S} \longrightarrow S$ corresponding respectively to $\mathcal{L}$ and $\mathcal{O}_X$, assuming that $S$ is smooth over a field of characteristic $0$.

In the case where $X/S$ has relative dimension 1, but where the morphism is allowed mild singularities, the problem of computing the class of the double ramification cycle was posed by Eliashberg in 2001, and solved by Bae, Holmes, Pandharipande, Schmitt and Schwarz in 2021 \cite{bae2023pixton} building on work of Janda, Pixton, Pandharinde and Zvonkine \cite{Janda_2017} when $\mathcal{L} \simeq \mathcal{O}(\sum_i a_i)$. The case of curves has found many applications in enumerative geometry and integrable systems (see for example \cite{Bae_2020, Buryak_2019, Clader_2017, Fan_2020, Farkas_2016, Oberdieck_2018, 10.1007/978-3-319-94881-2_9, schmitt2018dimension, Tseng_2020, blot2024strong, Buryak_Rossi_2023, cavalieri2016hurwitz, molcho2024case, van2021gromov}). A natural next step would be to compute DR cycles in higher relative dimensions, starting with the smooth case, in the hopes that they too have applications in counting higher dimensional objects (see \cite{bae2025countingsurfacescalabiyau4folds} for an introduction on this). Our main result is the following:
\begin{theorem}\label{thm:dr1}
    Let $S$ be a smooth connected scheme over a field of characteristic $0$, and $\pi: X\longrightarrow S$ a smooth projective morphism of schemes. Fix a relatively ample line bundle $\mathcal{F}$ on $X/S$. Then for any algebraically trivial line bundle $\mathcal{L}$ on $X/S$:
    \begin{equation*}
        \on{DR}(\mathcal{L}) = \dfrac{1}{d}\cdot\dfrac{(-\pi_*(c_1(\mathcal{L})^{.2}\cdot c_1(\mathcal{F})^{.(n-1)}))^{.g}}{g!} \in \on{CH}_{\mathbb{Q}}^g(S),
    \end{equation*}
    where $n = \dim_S X$, $g = \dim_S\Pic$, and $d\in \mathbb{Z}_{>0}$ is the rank of a vector bundle on $S$ which we'll define in Section \ref{sect2}.
\end{theorem}
\begin{remark}
    This formula is also true for numerically trivial line bundles, since they are algebraically trivial up to torsion (see Section \ref{sect2}). Theorem \ref{thm:dr1} also generalizes to the case where $S$ is a smooth connected algebraic stack over a field of characteristic $0$, and $\pi: X\longrightarrow S$ is a smooth, proper morphism of algebraic stacks that is representable by schemes (see Section \ref{sect3} for details).
\end{remark}
\begin{remark}
    If $X/S$ is a curve, we recover Hain's formula in the smooth case \cite{MR3184171}. If $X = C_1\times_SC_2\times_S...\times_S C_n$ is a product of curves, we know that $DR(\mathcal{L})$ is the intersection product of $DR(\mathcal{L}_i)$ where $\mathcal{L}_i$ is the pullback of $\mathcal{L}$ to $C_i$ via some section  (see \cite{holmes2019multiplicativity}), Theorem \ref{thm:dr1} in that case provides a formula that doesn't depend on the choice of sections (see Example \ref{exp:prod}).
\end{remark}
The starting point for the proof lies in observing that the jacobian of a curve $C/S$ is equipped with a canonical polarization (the theta divisor $\Theta$), and that this polarization is related to the universal line bundle $\mathcal U_C$ via the relation $p_{2*}(c_1(\mathcal{U}_C)^{.2}) = -2\Theta$ where $p_2: C\times_S \on{Pic}^0_{C/S}\longrightarrow \on{Pic}^0_{C/S}$ is the projection map. We can then use the Poincaré formula $\Theta^g = g! [e]$ to obtain the DR formula for (smooth) curves. In order to be able to apply the same principle in higher dimension, we first need to explicitly write down a polarization of $\Pic$. Fix a relatively ample line bundle $\mathcal{F}$ on $X/S$, and consider the divisor class $-p_{2*}(c_1(\mathcal{U}_X)^{.2}\cdot p_1^*c_1(\mathcal{F})^{.n-1})$, where $p_1, p_2$ are the projections of the product $X\times_S \on{Pic}^0_{X/S}$. We show by induction in Section \ref{sect2} that this gives a polarization of $\Pic$. The proof then proceeds as in the case of curves. 

\begin{acknowledgements}
    I would like to thank my advisor David Holmes for suggesting this problem and for his continued patience and guidance throughout the preparation of this article. I am also grateful to Chris Peters and Dimitri Zvonkine for answering some questions I had in relation to Example \ref{exp:prod}, and to Aitor Iribar Lopez for a nice discussion which inspired Remark \ref{rk:alb}, among other things. Finally, I am very grateful to Georgios Politopoulos for the many discussions we had on DR cycles in Leiden University.
\end{acknowledgements}

\section{Main result and proof}\label{sect2}

In this section, $S$ is a smooth quasi-projective connected scheme over a field of characteristic $0$.

 Let us first fix the following conventions: If $A/S$ is an abelian scheme, we write $\hat A/S$ for the dual, fitting into a diagram
 $$\begin{tikzcd}
& A\times_S \hat{A} \arrow{r}{p_2} \arrow{d}[swap]{p_1}
& \hat{A} \arrow{d}{\tilde{\pi}} \\
& A \arrow[swap]{r}{\pi}
& S
\end{tikzcd}$$
We write $\mathcal P_A$ for the Poincaré bundle on $A \times_S \hat A$, and 
\begin{equation}
    F_{A}: \on{CH}(A)\longrightarrow \on{CH}(\hat A); z \mapsto (p_2)_*(\exp(c_1(\mathcal P_A)) \cdot p_1^*z)
\end{equation}
for the Fourier transform (see \cite{Deninger1991}).

Let $\pi: X\longrightarrow S$ be a smooth projective scheme of relative dimension $n$, admitting a section. For any fibrewise algebraically trivial line bundle $\mathcal{L}$ on $X$, define the double ramification cycle of $\mathcal{L}$:
$$\operatorname{DR}(\mathcal{L}) \coloneqq \sigma^*(e) = \pi_*(\sigma \cdot e) \in \on{CH}^g(S)$$
where $\sigma, e$ are the sections of $\on{Pic}^0_{X/S}$ corresponding to $\mathcal{L}$ and $\mathcal{O}_X$ respectively, and $g = \dim \Pic$. Note that $\Pic$ is smooth in this case, see Proposition 9.5.20 and Remark 9.5.21 of \cite{MR2222646}.

\begin{lemma}\label{lem:DR_scaling}
    For any $r\in \mathbb{Z}_{\geq 0}$, $\on{DR}(\mathcal{L}^{\otimes r}) = r^{2g}\on{DR}(\mathcal{L}).$
\end{lemma}
\begin{proof}
By definition
$$\on{DR}(\mathcal{L}^{\otimes r}) = (r\sigma)^*e = \sigma^*[r]^*e,$$
where $[r]: \on{Pic}^0_{X/S}\longrightarrow \on{Pic}^0_{X/S}$ is the map of multiplication by $r$. For simplicity of notation, let us denote $\on{Pic}^0_{X/S}$ by $\mathcal{J}$, and let $\hat{\mathcal{J}}$ be its dual abelian scheme, with Fourier transform $F_{\hat{\mathcal{J}}}: \on{CH}(\hat{\mathcal{J}})\longrightarrow \on{CH}(\mathcal{J})$. 
By \cite[Corollary 2.22]{Deninger1991}, we know that 
$$e = (-1)^gF_{\hat{\mathcal{J}}}([\hat{\mathcal{J}}]).$$
Applying Proposition 2.16 of \cite{Deninger1991}, we obtain $$[r]^*e = r^{2g}e,$$
for any positive integer $r$, which finishes the proof.
\end{proof}
\begin{remark}
    This lemma allows us to extend the definition of $\on{DR}(\mathcal{L})$ (in the Chow ring of $S$ with $\mathbb{Q}$-coefficients) to the case where $\mathcal{L}$ is fibrewise numerically trivial, since numerically trivial line bundles are algebraically trivial up to torsion. 
\end{remark}
\begin{proposition}\label{prop:AS}
    If $A/S$ is an abelian scheme of dimension $g$, then 
    $$\on{DR}(\mathcal{L}) = (-1)^g\cdot\dfrac{\pi_*(c_1(\mathcal{L})^{.2g})}{(2g)!}.$$
\end{proposition}
\begin{proof}
    We again use the fact that $e_{\hat{A}} = (-1)^gF_{A}([A]),$ (Corollary 2.22 of \cite{Deninger1991}). Expanding this we get:
    \begin{align*}
        e_{\hat{A}} &= (-1)^g.p_{2*}(\exp{(c_1(\mathcal{P}_A))}.p_1^*([A]))\\
        &= (-1)^g.p_{2*}\left(\sum_{i\geq0} \dfrac{c_1(\mathcal{P}_A)^{.i}}{i!}\right)\\
        &= (-1)^g.p_{2*}\left(\dfrac{c_1(\mathcal{P}_A)^{.2g}}{(2g)!}\right),
    \end{align*}
    pulling everything back via $\sigma$ gives the desired formula.
\end{proof}

Let us go back to the setting where $X/S$ is a smooth projective scheme. Choose a relatively ample line bundle $\mathcal{F}$ on $X$, and consider the following diagram:
$$\begin{tikzcd}
& X\times_S \on{Pic}^0_{X/S} \arrow{r}{p_2} \arrow{d}[swap]{p_1}
& \on{Pic}^0_{X/S} \arrow{d}{\Tilde{\pi}} \\
& X \arrow[swap]{r}{\pi}
& S
\end{tikzcd}$$
Define $$E \coloneqq -{p_2}_*(c_1(\mathcal{U})^{.2}.p_1^*(c_1(\mathcal{F}))^{.(n-1)})\in \on{CH}^1(Pic^0_{X/S}),$$ where $\mathcal{U}$ is the universal line bundle of $X$ rigidified along some section $a \colon S\longrightarrow X$.

\begin{lemma}
    $e^*E = 0$ in $\on{CH}^1(S)$. 
\end{lemma}
\begin{proof}
We have that
    \begin{align*}
        e^*E &= -e^*p_{2_*}(c_1(\mathcal{U})^{.2}.p_1^*c_1(\mathcal{F})^{.(n-1)})\\
        &= -\pi_*\Tilde{e}^*(c_1(\mathcal{U})^{.2}.p_1^*c_1(\mathcal{F})^{.(n-1)}),
    \end{align*}
    where $\Tilde{e}$ is the induced map in the following diagram:
    $$\begin{tikzcd}
& X \arrow{r}{\Tilde{e}} \arrow{d}[swap]{\pi}
& X\times_S \mathcal{J} \arrow{d}{p_2} \\
& S \arrow[swap]{r}{e}
& \mathcal{J}
\end{tikzcd}$$
and so
    \begin{align*}
       e^*E &= -\pi_*((\Tilde{e}^*c_1(\mathcal{U}))^{.2}.\Tilde{e}^*p_1^*c_1(\mathcal{F})^{.(n-1)})\\
        &= 0 \text{ (by the universal property of } \mathcal{U}). \qedhere
    \end{align*}
\end{proof}

Let us denote $\Pic$ by $\mathcal{J}$, and choose a Cartier divisor $\mathbb{E}$ on $\mathcal{J}$ representing the class $E$. Since $e^*E$ vanishes in $\on{CH}^1(S)$, we know that $e^*\mathbb{E}$ is principal, let us choose a rational function $f$ such that $e^*\mathbb{E} = \on{div}f$. Define $\mathcal{E}\coloneqq \mathcal{O}(\mathbb{E})$, with rigidification given by $f^{-1}:e^*\mathcal{E}\longrightarrow \mathcal{O}_S$. 

\begin{theorem}\label{thm:DR}
    For any numerically trivial line bundle $\mathcal{L}$ on $X$:
    $$\on{DR}(\mathcal{L}) = \dfrac{1}{d}\cdot\dfrac{\left(-\pi_*(c_1(\mathcal{L})^{.2}.c_1(\mathcal{F})^{.(n-1)}\right)^{.g}}{g!},$$
    where $d := \on{rk} \tilde{\pi}_*\mathcal{E}\in \mathbb{Z}_{> 0}$, and $\tilde{\pi}$ denotes the structure map $\mathcal{J}\longrightarrow S$.
\end{theorem}
\noindent To prove this, we will need the following lemma:
\begin{lemma}\label{lem:AM}
$E$ is a relatively ample divisor on $\mathcal{J}$.
\end{lemma}
\begin{proof}
    We will do this by induction on the dimension $n$ of $X$. If $n = 1$, $E = 2\Theta$ where $\Theta$ is the canonical polarization of the Jacobian (see \cite{AST_1985__127__29_0}), which is ample.
    
Suppose the hypothesis is true for schemes of dimension $n-1$. Let $H\subset X$ be a smooth hypersurface such that $[H] = c_1(\mathcal{F})\in \on{CH}^1(X)$ (this exists since we can assume $\mathcal{F}$ to be very ample). Let $\iota: H\xhookrightarrow{} X$ be the inclusion map, and $\Tilde{\iota}: \mathcal{J}_X\longrightarrow \mathcal{J}_H$ be the pullback map by $\iota$. Consider the following commutative diagram:
\[
\begin{tikzcd}[row sep=2em, column sep=2.5em]
  & X\times_S \mathcal{J}_X \arrow[rr, "p_{2,X}"] & & \mathcal{J}_X \arrow[dd, "\Tilde{\iota}"] \\
H\times_S \mathcal{J}_X \arrow[ur, "\iota"] \arrow[dr, "\Tilde{\iota}"'] & & & \\
  & H\times_S \mathcal{J}_H \arrow[rr, "p_{2,H}"'] & & \mathcal{J}_H
\end{tikzcd}
\]
Let $\mathcal{U}_X$ and $\mathcal{U}_H$ be the universal line bundles of $X$ and $H$ respectively. Notice (by definition of $\iota$ and $\Tilde{\iota}$) that
$$\Tilde{\iota}^*(c_1(\mathcal{U}_H)^{.2}.p_1^*c_1(\mathcal{F}_{|H})^{.(n-2)}) = \iota^*c_1(\mathcal{U}_X)^{.2}.p_1^*c_1(\mathcal{F})^{.(n-2)}$$
and hence
$$\iota_*\Tilde{\iota}^*(c_1(\mathcal{U}_H)^{.2}.p_1^*c_1(\mathcal{F}_{|H})^{.(n-2)}) = c_1(\mathcal{U}_X)^{.2}.p_1^*c_1(\mathcal{F})^{.(n-2)}.p_1^*[H] = c_1(\mathcal{U}_X)^{.2}.p_1^*c_1(\mathcal{F})^{.(n-1)},$$
thus, by applying Theorem 6.2 of \cite{MR1644323}, we have:
$$p_{2,X_*}(c_1(\mathcal{U}_X)^{.2}.p_1^*c_1(\mathcal{F})^{.(n-1)}) = \Tilde{\iota}^*p_{2,H_*}(c_1(\mathcal{U}_H)^{.2}.p_1^*c_1(\mathcal{F}_{|H})^{.(n-2)}).$$
Since $p_{2,H_*}(c_1(\mathcal{U}_H)^{.2}.p_1^*c_1(\mathcal{F}_{|H})^{.(n-2)})$ is relatively ample by assumption, and $\Tilde{\iota}$ is finite by Lefschetz' Hyperplane Theorem (see Lemma 2A7 of \cite{kleiman1968algebraic}), we conclude that $E$ is relatively ample.
\end{proof}
\begin{proof}[Proof of Theorem \ref{thm:DR}]
    We will first use the following lemma to write down a formula for $e_{\mathcal{J}}$ in terms of the universal line bundle of $X$: 
    \begin{lemma}\label{lem:PF}
        Let $\mathcal{E}$ be a relatively ample and symmetric line bundle on an abelian scheme $\pi: A\longrightarrow S$, that is rigidified along the zero section $e_A$ of $A/S$, then
    $$\dfrac{c_1(\mathcal{E})^{.g}}{g!} = d.e_A \in \on{CH}^g(A),$$
    where $d = \on{rk} \pi_*\mathcal{E}\in \mathbb{Z}_{>0}$ and $e_A$ (by abuse of notation) refers to the class $e_{A*}([S])$ in $\on{CH}^g(A)$. Moreover
    \begin{equation*}
        c_1(\mathcal{E})^{.g+k}=0
    \end{equation*}
    for any $k >0$. 
    \end{lemma}
    \begin{proof}
        This is a corollary of Poincaré's formula (see Section 16.5.6 in \cite{MR2062673}), the proof in the case of abelian schemes follows the same lines.
    \end{proof}
    Let us prove that the line bundle $\mathcal{E}$ (or equivalently the divisor class $E$) is symmetric:
    \begin{align*}
        [-1]^*E &= -[-1]^*p_{2_*}(c_1(\mathcal{U})^{.2}.p_1^*c_1(\mathcal{F})^{.(n-1)})\\
        &= -p_{2_*}[-1]_X^*(c_1(\mathcal{U})^{.2}.p_1^*c_1(\mathcal{F})^{.(n-1)}),
    \end{align*}
    this follows from the following cartesian diagram:
    $$\begin{tikzcd}
& X\times_S \mathcal{J} \arrow{r}{[-1]_X} \arrow{d}[swap]{p_2}
& X\times_S \mathcal{J} \arrow{d}{p_2} \\
& \mathcal{J} \arrow[swap]{r}{[-1]}
& \mathcal{J}
\end{tikzcd}$$
so
    \begin{align*}
        [-1]^*E &= -p_{2_*}([-1]_X^*(c_1(\mathcal{U})^{.2}).[-1]_X^*p_1^*c_1(\mathcal{F}^{.(n-1)})),
    \end{align*}
    and since $[-1]_X^*\mathcal{U} = \mathcal{U}^{\otimes -1}$ (this follows from the universal property of $\mathcal{U}$), and $p_1\circ [-1]_X = p_1$, we conclude that $[-1]^*E = E$, so $\mathcal{E}$ is symmetric.

    Since $\mathcal{E}$ is ample, symmetric and $e^*\mathcal{E}\simeq\mathcal{O}_S$, we obtain by Lemma \ref{lem:PF} that $\dfrac{1}{d}\dfrac{E^{.g}}{g!} = e,$ so
    \begin{align*}
        \DR(\mathcal{L}) &= \sigma^*e\\
        &= \dfrac{1}{d}\dfrac{(\sigma^*\mathcal{E})^{.g}}{g!}\\
        &= \dfrac{1}{d}\dfrac{(-\sigma^*{p_2}_*(c_1(\mathcal{U})^{.2}.p_1^*(c_1(\mathcal{F}))^{.(n-1)}))^{.g}}{g!}\\
        &= \dfrac{1}{d}\dfrac{(-\pi_*\Tilde{\sigma}^*(c_1(\mathcal{U})^{.2}.p_1^*(c_1(\mathcal{F}))^{.(n-1)}))^{.g}}{g!}\\
        &= \dfrac{1}{d}\dfrac{(-\pi_*(c_1(\mathcal{L})^{.2}.c_1(F)^{.(n-1)}))}{g!},
    \end{align*}
    here $\Tilde{\sigma}$ is the induced map in the following cartesian diagram:

$$\begin{tikzcd}
& X\times_S \mathcal{J} \arrow{r}{p_2} \arrow{d}[swap]{p_1}
& \mathcal{J} \arrow{d}[swap]{\tilde{\pi}} \\
& X \ar[u, bend right, swap,"\sigma"]\arrow[swap]{r}{\pi}
& S \ar[u, bend right, swap, "\tilde{\sigma}"]
\end{tikzcd}$$
\end{proof}
\begin{remark}
     Theorem \ref{thm:DR} remains true in the more general case where $S$ is a smooth algebraic stack over a field $k$ of characteristic $0$, and $\pi: X\longrightarrow S$ is a smooth, proper morphism of algebraic stacks that is representable by schemes. See Section \ref{sect3}.
\end{remark}
\begin{remark}\label{rk:alb}
     In fact, taking inspiration from Proposition \ref{prop:AS}, we see that for a general scheme $X$ we also have that
     $$e_{\mathcal{J}} = (-1)^g.q_{2*}\left(\dfrac{c_1(\mathcal{P}_{\mathcal{J}})^{.2g}}{(2g)!}\right),$$
     where $q_2: \mathcal{J}\times_S \mathcal{J}^{\vee}\longrightarrow \mathcal{J}^{\vee}$ is the projection map. Then pulling everything back by $\sigma$ we get another formula:
     $$\on{DR}(\mathcal{L}) = (-1)^g\cdot \dfrac{\Bar{\pi}_*(c_1(\Bar{\mathcal{L}})^{.2g})}{(2g)!},$$
     where $\Bar{\pi}: \on{Alb}(X)\longrightarrow S$ is the structure morphism of the Albanese variety $\on{Alb}(X)$, and $\Bar{\mathcal{L}}$ is the line bundle on $\on{Alb}(X)\simeq \mathcal{J}^{\vee}$ corresponding to the section $\sigma: S\longrightarrow \mathcal{J}\simeq \mathcal{J}^{\vee \vee}$.
\end{remark}
    \begin{example}\label{exp:cur}
    If $\pi: C\longrightarrow S$ is a smooth curve of genus $g$, then $E = 2\Theta$ where $\Theta$ is the theta divisor on $\mathcal{J}$ (see \cite{AST_1985__127__29_0}), which moreover induces a principal polarization of $\mathcal{J}$. 
    Then for any numerically trivial line bundle $L$ on $C$, the formula of Theorem \ref{thm:DR} becomes
    $$\DR(\mathcal{L}) = \dfrac{(-\frac{1}{2}\pi_*(c_1(\mathcal{L})^{.2}))^{.g}}{g!}.$$
    This is a corollary of Hain's formula for compact type curves, see \cite{MR3184171}.
\end{example}
\begin{example}\label{exp:prod}
Let $\pi: X\coloneqq C_1\times_SC_2\longrightarrow S$ be a product of curves $\pi_i:C_i\longrightarrow S$ of respective genera $g_i$. Assume each curve $C_i\longrightarrow S$ admits a section $s_i: S\longrightarrow C_i$.\\
In order to be able to apply Theorem \ref{thm:DR}, we need to choose a relatively ample line bundle on $X/S$, a natural candidate could be the canonical bundle $\omega_X$, which is ample if $g_1, g_2\geq 2$. In that case we find that for any numerically trivial line bundle $\mathcal{L}$:
\begin{equation*}
\DR(\mathcal{L}) = \dfrac{1}{2^{g_1+g_2}(2g_1-2)^{g_2}\cdot(2g_2-2)^{g_1}}\cdot\dfrac{\left(-\pi_*(c_1(\mathcal{L})^{.2}.c_1(\omega_X))\right)^{.g_1+g_2}}{(g_1+g_2)!}.
\end{equation*}
\begin{proof}
    We know from Theorem \ref{thm:DR} that 
    $$\DR(\mathcal{L}) = \dfrac{1}{d_{\omega_X}}\cdot\dfrac{\left(-\pi_*(c_1(\mathcal{L})^{.2}.c_1(\omega_X))\right)^{.g_1+g_2}}{(g_1+g_2)!},$$
    however computing the rank $d_{\omega_X}$ explicitely is not trivial, we'll follow a different approach.
    Consider the following cartesian diagram:
$$\begin{tikzcd}
& C_1\times_S C_2 \arrow{r}[swap]{p_2} \arrow{d}{p_1}
& C_2 \ar[l, bend right, swap, "\tilde{s}_1"] \arrow{d}[swap]{\pi_2} \\
& C_1 \ar[u, bend left,"\tilde{s}_2"]\arrow{r}{\pi_1}
& S \ar[u, swap, bend right, "s_2"] \ar[l, bend left, "s_1"]
\end{tikzcd}$$
    where $\Tilde{s}_i$ is the map induced by $s_i$. Any numerically trivial line bundle $\mathcal{L}$ on $X$ decomposes uniquely into $\mathcal{L}\simeq p_1^*\mathcal{L}_1\otimes p_2^*\mathcal{L}_2$, where $\mathcal{L}_1$ and $\mathcal{L}_2$ are numerically trivial line bundles on $C_1$ and $C_2$ respectively. We can see this by taking $\mathcal{L}_1 = \tilde{s}_2^*\mathcal{L}$ and $\mathcal{L}_2 = \tilde{s}_1^*\mathcal{L}$.
    Then for any such $\mathcal{L}$:
    \begin{equation*}
        \DR(\mathcal{L}) = \sigma^*e = (\sigma_1, \sigma_2)^*(e_1, e_2) = \sigma_1^*e\cdot \sigma_2^*e_2 = \DR(\mathcal{L}_1)\cdot \DR(\mathcal{L}_2),
    \end{equation*}
    where $\sigma_i$ denotes the section of $\mathcal{J}_{C_i}\coloneqq \on{Pic}^0_{C_i/S}$ corresponding to $\mathcal{L}_i$, and $e_i$ is the zero section of $\mathcal{J}_{C_i}$ (note that $\mathcal{J}\coloneqq \Pic\simeq \mathcal{J}_{C_1}\times_S\mathcal{J}_{C_2}$).
    Using the DR formula for smooth curves (Example \ref{exp:cur}), we obtain:
    \begin{align*}
        \DR(\mathcal{L}) &= \dfrac{(-\frac{1}{2}\pi_{1*}(c_1(\mathcal{L}_1)^{.2}))^{.g_1}}{g_1!}\cdot \dfrac{(-\frac{1}{2}\pi_{2*}(c_1(\mathcal{L}_2)^{.2}))^{.g_2}}{g_2!}\\
        &= \dfrac{(-1)^{g_1+g_2}}{2^{g_1+g_2}.g_1!g_2!}\cdot (\pi_{1*}(c_1(\mathcal{L}_1)^{.2}))^{.g_1}\cdot (\pi_{2*}(c_1(\mathcal{L}_2)^{.2}))^{.g_2}.
    \end{align*}
    Our next step will be to expand $\left(\pi_*(c_1(\mathcal{L})^{.2}.c_1(\omega_X))\right)^{.g_1+g_2}$, and deduce $d_{\omega_X}$ by identification with the formula above.
    
    For convenience, let us denote by $L_i$ the class $c_1(\mathcal{L}_i)$, and by $D_i$ the class $c_1(\omega_{C_i})$.\\
    First, note that for any cycles $\alpha, \beta$ on $C_1$ and $C_2$ respectively:
    \begin{align*}
        \pi_*(p_1^*\alpha\cdot p_2^*\beta)&= (\pi_1\circ p_1)_*(p_1^*\alpha\cdot p_2^*\beta)\\
        \text{(projection formula)}&= \pi_{1*}(\alpha\cdot p_{1*}p_2^*\beta)\\
        &= \pi_{1*}(\alpha\cdot \pi_1^*\pi_{2*}\beta)\\
        \text{(projection formula)} &= \pi_{1*}\alpha\cdot \pi_{2*}\beta.
    \end{align*}
    Using this, and the fact that $\mathcal{L}\simeq p_1^*\mathcal{L}_1\otimes p_2^*\mathcal{L}$ and $\omega_X\simeq p_1^*\omega_{C_1}\otimes p_2^*\omega_{C_2}$, we obtain when expanding $\pi_*(c_1(\mathcal{L})^{.2}.c_1(\omega_X))$:
    \begin{align*}
        \pi_*(c_1(\mathcal{L})^{.2}.c_1(\omega_X)) &= \pi_{1*}L_1^{.2}\cdot\pi_{2*}D_2 + \pi_{1*}D_1\cdot \pi_{2*}L_2^{.2}\\
        &= (2g_2-2)\pi_{1*}L_1^{.2} + (2g_1-2)\pi_{2*}L_2^{.2},
    \end{align*}
    the remaining terms vanish due to dimension or degree reasons. We can now raise everything to $g_1+g_2$:
    \begin{equation*}
        \left(\pi_*(c_1(\mathcal{L})^{.2}.c_1(\omega_X))\right)^{g_1+g_2} = \sum_{k=0}^{g_1+g_2} {{g_1+g_2}\choose k} (2g_2-2)^k(2g_1-2)^{g_1+g_2-k}(\pi_{1*}L_1^{.2})^{.k}\cdot(\pi_{2*}L_2^{.2})^{.g_1+g_2-k}.
    \end{equation*}
    Since the $(\pi_{i*}L_i^{2})^{.k}$ vanish when $k> g_i$ (due to Lemma \ref{lem:PF}), we get:
    \begin{equation*}
        \left(\pi_*(c_1(\mathcal{L})^{.2}.c_1(\omega_X))\right)^{g_1+g_2} = {g_1+g_2\choose g_1}(2g_2-2)^{g_1}(2g_1-2)^{g_2}(\pi_{1*}L_1^{.2})^{.g_1}\cdot(\pi_{2*}L_2^{.2})^{.g_2}.
    \end{equation*}
    By identifying with the previous formula for $\DR(\mathcal{L})$, we conclude that:
    \begin{equation*}
\DR(\mathcal{L}) = \dfrac{1}{2^{g_1+g_2}(2g_1-2)^{g_2}\cdot(2g_2-2)^{g_1}}\cdot\dfrac{\left(-\pi_*(c_1(\mathcal{L})^{.2}.c_1(\omega_X))\right)^{.g_1+g_2}}{(g_1+g_2)!}.
\end{equation*}
\end{proof}

We can also choose to use the ample line bundle $\mathcal{F} \coloneqq \mathcal{O}(p_1^* [s_1(S)] + p_2^*[s_2(S)])$ instead of the canonical bundle, in which case we get the following formula:

$$\DR(\mathcal{L}) = \dfrac{(-\frac{1}{2}\pi_*(c_1(\mathcal{L})^{.2}.c_1(\mathcal{F})))^{.g_1+g_2}}{(g_1+g_2)!}.$$

The proof of this follows the same lines as before, we just need to compute the rank of the vector bundle $\tilde{\pi}_*\mathcal{E}$. This can be done by direct computation, or again by using that $\DR(\mathcal{L}) = \DR(\mathcal{L}_1)\cdot \DR(\mathcal{L}_2)$, expanding $\pi_*(c_1(\mathcal{L})^{.2}.c_1(\mathcal{F})))^{.g_1+g_2}$ and identifying the formulae in the end.
\end{example}

\section{Generalization to stacks}\label{sect3}
Let us now assume that $S$ is a smooth connected algebraic stack over a field $k$ of characteristic $0$.

Let $\pi: X\longrightarrow S$ be a smooth, proper morphism of algebraic stacks of relative dimension $n$ that is representable by schemes.
\begin{definition}
    A line bundle $\mathscr{F}$ on $X/S$ is said to be relatively ample if for any map $U\longrightarrow S$ from an affine scheme $U/k$, the pullback of $\mathscr{F}$ to $X\times_SU$ is a relatively ample line bundle on $X\times_SU\longrightarrow U$.
\end{definition}
Suppose that there exists a relatively ample line bundle $\mathscr{F}$ on $X/S$.\\
For any algebraically trivial line bundle $\mathscr{L}$ on $X$ we define the double ramification cycle of $\mathscr{L}$ as follows:
$$\DR(\mathscr{L}) \coloneqq \sigma^!e \in \on{CH}_{\OP}^g(S),$$ 
where $\on{CH}^*_{\OP}(S)$ denotes the operational Chow groups of $S$, as defined in \cite{MR4430955}, $g$ is the relative dimension of  the Picard algebraic space $\mathcal{J}\coloneqq\on{\mathcal{P}ic}^0_{X/S}$ over $S$, and $\sigma, e$ are the sections of $\mathcal{J}\longrightarrow S$ corresponding to $\mathscr{L}$ and $\mathscr{O}_X$ respectively.
\begin{remark}
    Here again we can extend the definition to numerically trivial line bundles.
\end{remark}
As we did with schemes, to compute $\DR(\mathscr{L})$ we first define a cycle class $E$ on $\mathcal{J}$ in the following way: 
$$E \coloneqq -{p_2}_*(c_1(\mathscr{U})^{.2}.p_1^*(c_1(\mathscr{F}))^{.(n-1)})\in \on{CH}^1_{\OP}(\mathcal{J}),$$ where $\mathscr{U}$ is the universal line bundle of $X$ rigidified along some section $a \colon S\longrightarrow X$.

\begin{lemma}\label{lem:stackAM}
    There exists a relatively ample, symmetric line bundle $\mathcal{E}$ on $\mathcal{J}$ and an isomorphism $\varphi:e^*\mathcal{E}\xrightarrow{\sim} \mathcal{O}_S$, where $e: S\longrightarrow \mathcal{J}$ is the zero section, such that $c_1(\mathcal{E}) = E$.
\end{lemma}
\begin{proof}
    We define $\mathcal{E}$ locally on $\mathcal{J}$:
    
    Let $U\longrightarrow S$ be a smooth map from a scheme $U/k$, since $S$ is smooth, $U$ is by definition a smooth scheme over $k$. Let $E_U$ be the following divisor class on $\mathcal{J}_U\coloneqq \mathcal{J}_{X\times_SU}:$
    \begin{equation*}
        E_U\coloneqq -p_{2*}(c_1(\mathscr{U})^{.2}.p_1^*c_1(\mathscr{F}_U)^{.(n-1)})\in \on{CH}^1(\mathcal{J}_U).
    \end{equation*}
    Choose a rigidified line bundle $(\mathcal{E}_U, \varphi_U: e_U^*\mathcal{E}_U\xlongrightarrow{\sim} \mathcal{O}_U)$ whose divisor class is $E_U$ (just like we did in the paragraph before Theorem \ref{thm:DR}).
    
    Let $U\longrightarrow S$ and $V\longrightarrow S$ be two smooth covers of $S$ by schemes. By a diagram chase one can show that the rigidified line bundles $p_U^*(\mathcal{E}_U)$ and $p_V^*(\mathcal{E}_V)$ are uniquely isomorphic, where $p_U: \mathcal{J}_{U\times_SV}\longrightarrow \mathcal{J}_U$ and $p_V:\mathcal{J}_{U\times_SV}\longrightarrow \mathcal{J}_V$ are the canonical maps. This shows that we can glue the line bundles $\mathcal{E}_U$ uniquely to obtain a rigidified line bundle $\mathcal{E}$ on $\mathcal{J}$, and since each $\mathcal{E}_U$ is relatively ample (by Lemma \ref{lem:AM}) and symmetric, the line bundle $\mathcal{E}$ is relatively ample over $\mathcal{J}$ and symmetric as well.
\end{proof}

\begin{theorem}\label{thm:stackDR}
    For any numerically trivial line bundle $\mathscr{L}$ on $X$, we have:
    \begin{equation}\label{DReq}
        \DR(\mathscr{L}) = \dfrac{1}{d}.\dfrac{\left(-\pi_*(c_1(\mathscr{L})^{.2}.c_1(\mathscr{F})^{.(n-1)}\right)^{.g}}{g!},
    \end{equation}
    where $d := \on{rk} \tilde{\pi}_*\mathcal{E}\in \mathbb{Z}_{> 0}$.
\end{theorem}
\begin{proof}
    The key idea is to show that 
    \begin{equation}\label{AMeq}
        \dfrac{E^{.g}}{g!} = d.e
    \end{equation}
    in $\on{CH}_{\OP}^g(\mathcal{J})$. We obtain the desired formula by pulling \ref{AMeq} back to $S$ via $\sigma$.
    
    Let us start by looking at the moduli stack $\mathcal{N}_{g, d}$ representing the functor which to a scheme $T$ associates the groupoid of elements $(A_T, L_T, \phi_T)$ where:
    \begin{itemize}
        \item $\pi_T:A_T\longrightarrow T$ is an abelian scheme of dimension $g$.
        \item $L_T$ is a relatively ample and symmetric line bundle on $A_T/T$, such that $\on{rk}\; \pi_{T*}L_T = d$.
        \item $\phi_T$ is the choice of an isomorphism $e_T^*L_T \xrightarrow{\sim} \mathcal{O}_T$, where $e_T:T\longrightarrow A_T$ is the zero section of $A_T$.
    \end{itemize}
    There is a canonical map $\varphi: \mathcal{N}_{g, d}\longrightarrow \mathcal{A}_{g, d}$, where $\mathcal{A}_{g, d}$ is the moduli space of abelian varieties of dimension $g$ with a polarization of degree $d^2$. In fact, $\mathcal{N}_{g, d}$ is a principal homogeneous space over $\mathcal{A}_{g, d}$, under the group of $2$-torsion points of the dual universal abelian variety $\chi_{g, d}^{\vee}[2]$ (see Section 7, Chapter 4 of \cite{MR1083353}), with the action given on the fibre of a $T$-point $(A, \varphi: A\longrightarrow A^{\vee})$ by:
    \begin{align*}
        \chi_{g, d}^{\vee}[2](A,\varphi) \times \mathcal{N}_{g, d}(A,\varphi) &\longrightarrow \mathcal{N}_{g, d}(A, \varphi)\\
        (F, L) &\longmapsto F\otimes L.
    \end{align*}
    
    \noindent In particular, the map $\varphi$ is finite étale; and since $\mathcal{A}_{g, d}$ is separated and of finite type over $k$, we conclude that $\mathcal{N}_{g, d}$ is separated and of finite type over $k$ as well.
    
    Let $\chi\longrightarrow \mathcal{N}_{g, d}$ denote the universal abelian variety, and $\mathcal{L}$ the universal line bundle. By Theorem 1.1 of \cite{OLSSON200593}, there exists a proper surjective map $U\longrightarrow \mathcal{N}_{g, d}$ from a quasi-projective scheme over $k$, which we may also assume is smooth by resolution of singularities (See \cite{da38f13a-4476-31dd-93dc-ce446d771081}).
    Let $\chi_U$ denote the fibre product $\chi\times_{\mathcal{N}_{g, d}}U$, and $\rho:\chi_{U}\longrightarrow \chi$ be the induced map on abelian schemes.
    \begin{lemma}\label{lem: INJ}
        The pullback map $\rho^*:\on{CH}_\OP^*(\chi)\longrightarrow \on{CH}_{\OP}^*(\chi_U)$ is injective.
    \end{lemma}
    \begin{proof}[Proof of Lemma \ref{lem: INJ}]
    Let $\alpha$ be a class in $\on{CH}_{\OP}^*(\chi)$ such that $\rho^*\alpha=0$ in $\on{CH}_{\OP}^*(\chi_U)$. Let us prove that $\alpha = 0$.
    
    Let $\mathcal{V}\longrightarrow \chi$ be a morphism from a finite type algebraic space $\mathcal{V}$ over $k$, we need to show that $\alpha.Z = 0$ for any $Z\in \on{CH}^*(\mathcal{V})$ (where $\on{CH}^*(\mathcal{V})$ is defined as in \cite{MR4430955}).
    Since the map $U\longrightarrow \mathcal{N}_{g, d}$ is proper and surjective, it follows that $\rho$ is also proper and surjective, and so is the induced map $\Tilde{\rho}: \chi_U\times_{\chi}\mathcal{V}\longrightarrow \mathcal{V}$. By Proposition B.19 of \cite{MR4430955}, the pushforward map
    $$\Tilde{\rho}_*: \on{CH}_*(\chi_U\times_{\chi}\mathcal{V})\longrightarrow \on{CH}_*(\mathcal{V})$$
    is surjective, so for any $Z\in \on{CH}(\mathcal{V})$ there exists $X\in \on{CH}(\chi_U\times_{\chi}\mathcal{V})$ such that $Z = \Tilde{\rho}_*W$, then
    $$\alpha.Z = \alpha.\Tilde{\rho}_*W = \Tilde{\rho}_*(\rho^*\alpha.W) = 0.$$
    \end{proof}
    Since the pullback $\rho^*\mathcal{L}$ is ample, symmetric and rigidified along the zero section on $\chi_{U} \longrightarrow U$, which is an abelian scheme over $U$, we can apply Lemma \ref{lem:PF} in this case and obtain
    \begin{equation*}
        \dfrac{(\rho^*c_1(\mathcal{L}))^{.g}}{g!} = d.e_{\chi_{U}} = \rho^*(d.e_{\chi}).
    \end{equation*}
    This and Lemma \ref{lem: INJ} imply that 
    $$\dfrac{c_1(\mathcal{L})^{.g}}{g!} = d.e_{\chi}.$$
    By Lemma \ref{lem:stackAM}, the triple $(\mathcal{J}, \mathcal{E}, \varphi)$ corresponds to a map $S\longrightarrow \mathcal{N}_{g, d}$, and we have the following cartesian diagram:
    $$\begin{tikzcd}
& \mathcal{J} \arrow{r} \arrow{d}
& \chi \arrow{d} \\
& S \arrow{r}
& \mathcal{N}_{g, d}
\end{tikzcd}$$
    then $\mathcal{E}$ is by definition the pullback of $\mathcal{L}$ via the map $\mathcal{J}\longrightarrow \chi$, which gives \ref{AMeq}. 
\end{proof}

\bibliography{refs}

@article{Deninger1991,
author = {Deninger, Christopher and Murre, Jacob},
journal = {Journal für die reine und angewandte Mathematik},
pages = {201-219},
title = {Motivic decomposition of abelian schemes and the {F}ourier transform.},
url = {http://eudml.org/doc/153379},
volume = {422},
year = {1991},
}

@incollection{AST_1985__127__29_0,
     author = {Moret-Bailly, Laurent},
     title = {M\'etriques permises},
     booktitle = {S\'eminaire sur les pinceaux arithm\'etiques : la conjecture de Mordell},
     series = {Ast\'erisque},
     pages = {29--87},
     year = {1985},
     publisher = {Soci\'et\'e math\'ematique de France},
     number = {127},
     mrnumber = {801918},
     zbl = {1182.11028},
     language = {fr},
     url = {https://www.numdam.org/item/AST_1985__127__29_0/}
}

@book {MR1644323,
    AUTHOR = {Fulton, William},
     TITLE = {Intersection theory},
    SERIES = {Ergebnisse der Mathematik und ihrer Grenzgebiete. 3. Folge. A
              Series of Modern Surveys in Mathematics [Results in
              Mathematics and Related Areas. 3rd Series. A Series of Modern
              Surveys in Mathematics]},
    VOLUME = {2},
   EDITION = {Second},
 PUBLISHER = {Springer-Verlag, Berlin},
      YEAR = {1998},
     PAGES = {xiv+470},
      ISBN = {3-540-62046-X; 0-387-98549-2},
   MRCLASS = {14C17 (14-02)},
  MRNUMBER = {1644323},
       DOI = {10.1007/978-1-4612-1700-8},
       URL = {https://doi.org/10.1007/978-1-4612-1700-8},
}

@book {MR2062673,
    AUTHOR = {Birkenhake, Christina and Lange, Herbert},
     TITLE = {Complex abelian varieties},
    SERIES = {Grundlehren der mathematischen Wissenschaften [Fundamental
              Principles of Mathematical Sciences]},
    VOLUME = {302},
   EDITION = {Second},
 PUBLISHER = {Springer-Verlag, Berlin},
      YEAR = {2004},
     PAGES = {xii+635},
      ISBN = {3-540-20488-1},
   MRCLASS = {14-02 (14H37 14Kxx 32G20)},
  MRNUMBER = {2062673},
MRREVIEWER = {Fumio\ Hazama},
       DOI = {10.1007/978-3-662-06307-1},
       URL = {https://doi.org/10.1007/978-3-662-06307-1},
}

@incollection {MR3184171,
    AUTHOR = {Hain, Richard},
     TITLE = {Normal functions and the geometry of moduli spaces of curves},
 BOOKTITLE = {Handbook of moduli. {V}ol. {I}},
    SERIES = {Adv. Lect. Math. (ALM)},
    VOLUME = {24},
     PAGES = {527--578},
 PUBLISHER = {Int. Press, Somerville, MA},
      YEAR = {2013},
      ISBN = {978-1-57146-257-2},
   MRCLASS = {14H10 (32G15)},
  MRNUMBER = {3184171},
}

@article {MR4430955,
    AUTHOR = {Bae, Younghan and Schmitt, Johannes},
     TITLE = {Chow rings of stacks of prestable curves {I}},
      NOTE = {With an appendix by Bae, Schmitt and Jonathan Skowera},
   JOURNAL = {Forum Math. Sigma},
  FJOURNAL = {Forum of Mathematics. Sigma},
    VOLUME = {10},
      YEAR = {2022},
     PAGES = {Paper No. e28, 47},
      ISSN = {2050-5094},
   MRCLASS = {14H10 (14C15 14C17)},
  MRNUMBER = {4430955},
MRREVIEWER = {Reinier\ Kramer},
       DOI = {10.1017/fms.2022.21},
       URL = {https://doi.org/10.1017/fms.2022.21},
}

@book {MR1083353,
    AUTHOR = {Faltings, Gerd and Chai, Ching-Li},
     TITLE = {Degeneration of abelian varieties},
    SERIES = {Ergebnisse der Mathematik und ihrer Grenzgebiete (3) [Results
              in Mathematics and Related Areas (3)]},
    VOLUME = {22},
      NOTE = {With an appendix by David Mumford},
 PUBLISHER = {Springer-Verlag, Berlin},
      YEAR = {1990},
     PAGES = {xii+316},
      ISBN = {3-540-52015-5},
   MRCLASS = {14K10 (11G10 14D20 14K25)},
  MRNUMBER = {1083353},
MRREVIEWER = {Min\ Ho\ Lee},
       DOI = {10.1007/978-3-662-02632-8},
       URL = {https://doi.org/10.1007/978-3-662-02632-8},
}

@article{OLSSON200593,
title = {On proper coverings of {A}rtin stacks},
journal = {Advances in Mathematics},
volume = {198},
number = {1},
pages = {93-106},
year = {2005},
note = {Special volume in honor of Michael Artin: Part I},
issn = {0001-8708},
doi = {https://doi.org/10.1016/j.aim.2004.08.017},
url = {https://www.sciencedirect.com/science/article/pii/S0001870805000423},
author = {Martin C. Olsson}
}

@article{da38f13a-4476-31dd-93dc-ce446d771081,
 ISSN = {0003486X, 19398980},
 URL = {http://www.jstor.org/stable/1970486},
 author = {Heisuke Hironaka},
 journal = {Annals of Mathematics},
 number = {1},
 pages = {109--203},
 publisher = {[Annals of Mathematics, Trustees of Princeton University on Behalf of the Annals of Mathematics, Mathematics Department, Princeton University]},
 title = {Resolution of Singularities of an Algebraic Variety Over a Field of Characteristic Zero: I},
 urldate = {2026-01-22},
 volume = {79},
 year = {1964}
}

@article{Janda_2017,
   title={Double ramification cycles on the moduli spaces of curves},
   volume={125},
   ISSN={1618-1913},
   url={http://dx.doi.org/10.1007/s10240-017-0088-x},
   DOI={10.1007/s10240-017-0088-x},
   number={1},
   journal={Publications mathématiques de l’IHÉS},
   publisher={Springer Science and Business Media LLC},
   author={Janda, F. and Pandharipande, R. and Pixton, A. and Zvonkine, D.},
   year={2017},
   month=may, pages={221–266} }

@article{bae2025countingsurfacescalabiyau4folds,
      title={Counting surfaces on {C}alabi-{Y}au 4-folds {I}: {F}oundations}, 
      author={Younghan Bae and Martijn Kool and Hyeonjun Park},
      year={2025},
      journal={arXiv preprint, arXiv:2208.09474},
      eprint={2208.09474},
      archivePrefix={arXiv},
      primaryClass={math.AG},
      url={https://arxiv.org/abs/2208.09474}, 
}

@article{Bae_2020,
   title={Tautological relations for stable maps to a target variety},
   volume={58},
   ISSN={1871-2487},
   url={http://dx.doi.org/10.4310/ARKIV.2020.v58.n1.a2},
   DOI={10.4310/arkiv.2020.v58.n1.a2},
   number={1},
   journal={Arkiv för Matematik},
   publisher={International Press of Boston},
   author={Bae, Younghan},
   year={2020},
   pages={19–38} }

@article{Buryak_2019,
   title={{DR}/{DZ} equivalence conjecture and tautological relations},
   volume={23},
   ISSN={1465-3060},
   url={http://dx.doi.org/10.2140/gt.2019.23.3537},
   DOI={10.2140/gt.2019.23.3537},
   number={7},
   journal={Geometry \& Topology},
   publisher={Mathematical Sciences Publishers},
   author={Buryak, Alexandr and Guéré, Jérémy and Rossi, Paolo},
   year={2019},
   month=dec, pages={3537–3600} }

@article{Clader_2017,
   title={Powers of the {T}heta Divisor and Relations in the Tautological Ring},
   volume={2018},
   ISSN={1687-0247},
   url={http://dx.doi.org/10.1093/imrn/rnx115},
   DOI={10.1093/imrn/rnx115},
   number={24},
   journal={International Mathematics Research Notices},
   publisher={Oxford University Press (OUP)},
   author={Clader, Emily and Grushevsky, Samuel and Janda, Felix and Zakharov, Dmitry},
   year={2017},
   month=jun, pages={7725–7754} }

@article{Fan_2020,
   title={Structures in genus-zero relative {G}romov–{W}itten theory},
   volume={13},
   ISSN={1753-8424},
   url={http://dx.doi.org/10.1112/topo.12131},
   DOI={10.1112/topo.12131},
   number={1},
   journal={Journal of Topology},
   publisher={Wiley},
   author={Fan, Honglu and Wu, Longting and You, Fenglong},
   year={2020},
   month=mar, pages={269–307} }

@article{Farkas_2016,
   title={THE MODULI SPACE OF TWISTED CANONICAL DIVISORS},
   volume={17},
   ISSN={1475-3030},
   url={http://dx.doi.org/10.1017/S1474748016000128},
   DOI={10.1017/s1474748016000128},
   number={3},
   journal={Journal of the Institute of Mathematics of Jussieu},
   publisher={Cambridge University Press (CUP)},
   author={Farkas, Gavril and Pandharipande, Rahul},
   year={2016},
   month=apr, pages={615–672} }

@article{Oberdieck_2018,
   title={Holomorphic anomaly equations and the {I}gusa cusp form conjecture},
   volume={213},
   ISSN={1432-1297},
   url={http://dx.doi.org/10.1007/s00222-018-0794-0},
   DOI={10.1007/s00222-018-0794-0},
   number={2},
   journal={Inventiones mathematicae},
   publisher={Springer Science and Business Media LLC},
   author={Oberdieck, Georg and Pixton, Aaron},
   year={2018},
   month=feb, pages={507–587} }

@InProceedings{10.1007/978-3-319-94881-2_9,
author="Pixton, Aaron",
editor="Christophersen, Jan Arthur
and Ranestad, Kristian",
title="Generalized {B}oundary {S}trata {C}lasses",
booktitle="Geometry of Moduli",
year="2018",
publisher="Springer International Publishing",
address="Cham",
pages="285--293",
isbn="978-3-319-94881-2"
}

@article{schmitt2018dimension,
  title={Dimension theory of the moduli space of twisted $ k $-differentials},
  author={Schmitt, Johannes},
  journal={Documenta Mathematica},
  volume={23},
  pages={871--894},
  year={2018}
}

@article{Tseng_2020,
   title={Higher genus relative and orbifold {G}romov–{W}itten
invariants},
   volume={24},
   ISSN={1465-3060},
   url={http://dx.doi.org/10.2140/gt.2020.24.2749},
   DOI={10.2140/gt.2020.24.2749},
   number={6},
   journal={Geometry \& Topology},
   publisher={Mathematical Sciences Publishers},
   author={Tseng, Hsian-Hua and You, Fenglong},
   year={2020},
   month=dec, pages={2749–2779} }

@article{blot2024strong,
  title={On the strong {DR}/{DZ} equivalence conjecture},
  author={Blot, Xavier and Lewanski, Danilo and Shadrin, Sergey},
  journal={arXiv preprint, arXiv:2405.12334},
  year={2024}
}

@article{Buryak_Rossi_2023, title={A GENERALISATION OF {W}ITTEN’S CONJECTURE FOR THE {P}IXTON CLASS AND THE NONCOMMUTATIVE {K}d{V} HIERARCHY}, volume={22}, DOI={10.1017/S1474748022000354}, number={6}, journal={Journal of the Institute of Mathematics of Jussieu}, author={Buryak, Alexandr and Rossi, Paolo}, year={2023}, pages={2987–3009}}

@article{cavalieri2016hurwitz,
  title={{H}urwitz theory and the double ramification cycle},
  author={Cavalieri, Renzo},
  journal={Japanese Journal of Mathematics},
  volume={11},
  number={2},
  pages={305--331},
  year={2016},
  publisher={Springer}
}

@article{molcho2024case,
  title={A case study of intersections on blowups of the moduli of curves},
  author={Molcho, Sam and Ranganathan, Dhruv},
  journal={Algebra \& Number Theory},
  volume={18},
  number={10},
  pages={1767--1816},
  year={2024},
  publisher={Mathematical Sciences Publishers}
}

@article{van2021gromov,
  title={{G}romov--{W}itten theory of {K}3 surfaces and a {K}aneko--{Z}agier equation for {J}acobi forms},
  author={van Ittersum, Jan-Willem and Oberdieck, Georg and Pixton, Aaron},
  journal={Selecta Mathematica},
  volume={27},
  number={4},
  pages={64},
  year={2021},
  publisher={Springer}
}

@article{holmes2019multiplicativity,
  title={Multiplicativity of the double ramification cycle},
  author={Holmes, David and Pixton, Aaron and Schmitt, Johannes},
  journal={Documenta Mathematica},
  volume={24},
  pages={545--562},
  year={2019},
  publisher={Deutsche Mathematiker-Vereinigung}
}

@book{kleiman1968algebraic,
  title={Algebraic cycles and the {W}eil conjectures},
  author={Kleiman, Steven L.},
  year={1968},
  publisher={Columbia university, Department of mathematics}
}

@article{bae2023pixton,
  title={Pixton’s formula and {A}bel--{J}acobi theory on the {P}icard stack},
  author={Bae, Younghan and Holmes, David and Pandharipande, Rahul and Schmitt, Johannes and Schwarz, Rosa},
  journal={Acta Mathematica},
  volume={230},
  number={2},
  pages={205--319},
  year={2023},
  publisher={Lehigh University Bethlehem, Penn., USA}
}

@book {MR2222646,
    AUTHOR = {Fantechi, Barbara and G\"ottsche, Lothar and Illusie, Luc and
              Kleiman, Steven L. and Nitsure, Nitin and Vistoli, Angelo},
     TITLE = {Fundamental algebraic geometry},
    SERIES = {Mathematical Surveys and Monographs},
    VOLUME = {123},
      NOTE = {Grothendieck's FGA explained},
 PUBLISHER = {American Mathematical Society, Providence, RI},
      YEAR = {2005},
     PAGES = {x+339},
      ISBN = {0-8218-3541-6},
   MRCLASS = {14-06 (14A15 14D15 14F20)},
  MRNUMBER = {2222646},
MRREVIEWER = {Liam\ O'Carroll},
       DOI = {10.1090/surv/123},
       URL = {https://doi.org/10.1090/surv/123},
}
\bibliographystyle{alpha}
\end{document}